\documentclass[a4paper,10pt]{article}
\usepackage{stmaryrd}
\usepackage{amsfonts}
\usepackage{bbm}
\usepackage{amscd}
\usepackage{mathrsfs}
\usepackage{latexsym,amssymb,amsmath,amscd,amscd,amsthm,amsxtra,xypic}
\usepackage[dvips]{graphicx}
\usepackage[utf8]{inputenc}
\usepackage[T1]{fontenc}
\usepackage{lmodern}
\usepackage{amssymb}
\usepackage[all]{xy}
\usepackage{nicefrac,mathtools,enumitem}
\usepackage{microtype}

\newtheorem{thm}{Theorem}[section]
\newtheorem{lem}[thm]{Lemma}\newtheorem{lemma}[thm]{Lemma}
\newtheorem{cor}[thm]{Corollary}
\newtheorem{pro}[thm]{Proposition}

\newtheorem{rmk}[thm]{Remark}\newtheorem{remark}[thm]{Remark}

\newtheorem{defi}[thm]{Definition}

\newcommand{\be }{\begin{equation}}
\newcommand{\ee }{\end{equation}}

\newcommand{\defbe}{\triangleq}
\newcommand{\pf}{\noindent{\bf Proof.}\ }

\newcommand{\R}{\mathbb R}\newcommand{\Z}{\mathbb Z}
\newcommand{\Real}{\mathbb R}

\newcommand{\huaC}{{\mathcal{C}}}

\newcommand{\huaK}{\mathcal{K}}

\newcommand{\huaT}{\mathcal{T}}

\newcommand{\CWM}{C^{\infty}(M)}

\newcommand{\set}[1]{\left\{#1\right\}}

\newcommand{\frkd}{\mathfrak d}

\newcommand{\frkg}{\mathfrak g}\newcommand{\g}{\mathfrak g}

\newcommand{\frkk}{\mathfrak k}

\newcommand{\frkr}{\mathfrak r}

\newcommand{\frkt}{\mathfrak t}

\newcommand{\frky}{\mathfrak y}

\newcommand{\frkL}{\mathfrak L}

\newcommand{\frkX}{\mathfrak X}

\def\qed{\hfill ~\vrule height6pt width6pt depth0pt}

\newcommand{\half}{\frac{1}{2}}

\newcommand{\pairing}[1]{\left\langle #1\right\rangle}
\newcommand{\ppairingE}[1]{\left ( #1\right )_E}

\newcommand{\conpairing}[1]{\left\langle  #1\right\rangle }
\newcommand{\Courant}[1]{\left\llbracket  #1\right\rrbracket }

\newcommand{\Lied}{\frkL}
\newcommand{\jet}{\mathfrak{J}}

\newcommand{\jetd}{\mathbbm{d}}
\newcommand{\dev}{\mathfrak{D}}

\newcommand{\pie}{^\prime}
\newcommand{\Id}{\rm{Id}}

\newcommand{\e}{\mathbbm{e}}
\newcommand{\p}{\mathbbm{p}}
\newcommand{\id}{\mathbbm{i}}

\newcommand{\dM}{\mathrm{d}}

\newcommand{\omni}{\mathcal{E}}
\newcommand{\omnirho}{\rho_{\varepsilon}}

\newcommand{\Hom}{\mathrm{Hom}}

\newcommand{\Ad}{\mathrm{Ad}}
\newcommand{\Aut}{\mathrm{Aut}}

\newcommand{\gl}{\mathfrak {gl}}

\newcommand{\Ker}{\mathrm{Ker}}

\newcommand{\End}{\mathrm{End}}
\newcommand{\ad}{\mathrm{ad}}

\newcommand{\inv}{\mathrm{inv}}

\newcommand{\ve}{\mathrm{v}}
\newcommand{\h}{\mathrm{h}}
\newcommand{\sgn}{\mathrm{sgn}}
\newcommand{\Ksgn}{\mathrm{Ksgn}}

\begin{document}
\title{
{Semidirect products of representations up to homotopy
\thanks
 {Supported by the German Research Foundation 
(Deutsche Forschungsgemeinschaft (DFG)) through 
the Institutional Strategy of the University of G\"ottingen
 }
} }
\author{Yunhe Sheng  \\
Department of Mathematics, Jilin University,
 Changchun 130012, Jilin, China
\\\vspace{3mm}
email: ysheng888@gmail.com\\
Chenchang Zhu\\
Courant Research Centre ``Higher Order Structures'', 
University of G\"ottingen\\
email:zhu@uni-math.gwdg.de}
\date{}


\maketitle
\begin{abstract}
We study the semidirect product of a Lie algebra  with a
representation up to homotopy and provide various examples coming
from Courant algebroids, string Lie 2-algebras, and omni-Lie
algebroids. In the end, we study the semidirect product of a Lie
group with a representation up to homotopy and use it to give an
integration of a certain string Lie 2-algebra.  \tableofcontents
\end{abstract}
\section{Introduction}

This paper is the first part of our project to integrate
representations up to homotopy of Lie algebras (algebroids). Our
original motivation is to integrate the standard Courant algebroid
$TM\oplus T^*M$ since it is this Courant algebroid that is much used
in Hitchin and Gualtieri's program of generalized complex geometry.
Courant algebroids are Lie 2-algebroids in the sense of Roytenberg
and \v{S}evera \cite{royt, s:funny}. The general procedure to
integrate Lie $n$-algebras (algebroids) is already described in
\cite{getzler, henriques, s:funny}. We want to pursue some explicit formulas
for the special case of the  standard Courant algebroid. It turns
out that the sections of the Courant algebroid $TM \oplus T^*M$ form
a semidirect product of a Lie algebra with a representation up to
homotopy.   Abad and Crainic \cite{abad-crainic:rep-homotopy}
recently studied the representations up to homotopy of Lie algebras,
Lie groups, and even Lie algebroids, Lie groupoids, in general. Just
as one can form the semidirect product of a Lie algebra with a
representation, one can form the semidirect product with
representations up to homotopy too. In our case, the semidirect
product coming from the standard Courant algebra is a Lie 2-algebra.
But using the fact that it is also a semidirect product, the
integration becomes easier. The integration result is related to the
semidirect product of Lie groups with its representation up to
homotopy, which will be discussed in Section \ref{sec:gp}. However
it turns out that the concept of representation up to homotopy of
Lie groups of Abad and Crainic will not be general enough to cover
all the integration results. This we will continue in a forthcoming
paper \cite{sheng-zhu:II}.

In this paper we focus on exhibiting more examples of representation
up to homotopy and  their semidirect products to demonstrate the
importance of our integration procedure. The examples are all sorts
of variations of Courant algebroids. One is Chen and Liu's omni-Lie
algebroids, which generalizes Weinstein's omni-Lie algebras. Hence
we expect to give an integration to Weinstein's omni-Lie algebras
via Lie 2-algebras in the next paper \cite{sheng-zhu:II}.

Another example comes from the so called string Lie 2-algebra. It is
essentially  a Courant algebroid over a point (see Section
\ref{sec:string}),  namely  a Lie algebra with an adjoint-invariant
inner product. This sort of Lie algebra is usually called a {\em
quadratic Lie algebra}. This concept also appears in the context of
Manin triples and double Lie algebras. The example  $\R \to \g
\oplus \g^*$ that we consider in this paper is an analogue of the
standard Courant algebroid, and is basically a special case taken
from \cite{lu-weinstein} \footnote{private conversation
  with Jiang-Hua Lu}. We give an integration of
the string Lie 2-algebra $\R \to \g \oplus \g^*$  at the end.

Usually people require the base Lie algebra of a string Lie algebra
to be semisimple and of compact type (see Remark
\ref{rk:semisimple}). For such usual sort of string Lie 2-algebras,
Baez et al \cite{baez:2gp} have proved a no-go theorem, namely such
string Lie 2-algebras can not be integrated to finite dimensional
semi-strict Lie 2-groups. Here a {\em semi-strict Lie 2-group} is a
group object in $\rm DiffCat$, where  $\rm DiffCat$ is the
2-category consisting of categories, functors, and natural
transformations in the category of differential manifolds, or
equivalently $\rm DiffCat$ is a 2-category with Lie groupoids as
objects, strict morphisms of Lie groupoids as morphisms, and
2-morphisms of Lie groupoids as 2-morphisms.  Our semi-strict Lie
2-group is actually called a Lie 2-group by the authors in
\cite{baez:2gp}. However, we call it a semi-strict Lie 2-group
because  compared to the Lie 2-group in the sense of Henriques
\cite{henriques}, or equivalently the stacky group in the sense of
Blohmann \cite{blohmann},  it is stricter. Basically,  their Lie
2-group is a group object in the 2-category with objects as Lie
groupoids, morphisms as Hilsum-Skandalas bimodules (or generalized
morphisms), 2-morphisms as 2-morphisms of Lie groupoids.
Schommer-Pries realizes the string 2-group as such a Lie 2-group
with a finite dimensional model \cite{schommer:string-finite-dim}
and the integration of a string Lie 2-algebra to such a model is a
work in progress \cite{ccc:integration-string}.

It is not needed in the
definition of the string Lie 2-algebra for the base Lie algebra to be
semisimple of compact type. One only needs a quadratic Lie algebra.
As soon as we relax this condition on compactness, we find out that
one can integrate  $\R \to \g \oplus \g^*$ to a finite dimensional
semi-strict Lie
2-group in the sense of Baez et al. The integrating object is actually
a special Lie 2-group (very close to a strict Lie 2-group) in the
sense of Baez et al.

Then of course, as we relax the
condition, we are in danger that the class corresponding to this
Lie 2-algebra in $H^3(\g\oplus \g^*, \R)$ might be trivial, and
consequently our Lie 2-algebra might be strict. Then what
we have done
would not have been a big surprise  because  a strict
Lie 2-algebra corresponds to a crossed module of Lie algebras, and it
easily integrates to a strict Lie 2-group by integrating the crossed module.  However,
we verified that when $\g$ itself (not $\g \oplus \g^*$) is semisimple, this Lie 2-algebra is not strict.

{\bf Acknowledgement:} We give warmest thanks to  Zhang-Ju Liu,
Jiang-Hua Lu, Giorgio Trentinagia and Marco Zambon for
useful comments and discussion. Y. Sheng gives his warmest thanks to Courant Research Centre ``Higher Order Structures'', G\"{o}ttingen University, where this work was
done when he visited there.

\section{Representations up to homotopy of Lie algebras}
In this section, we first consider the 2-term representation up to
homotopy of Lie algebras. We give explicit formulas of the
corresponding 2-term $L_\infty$-algebra, which is their semidirect
product. Then we give several interesting examples including Courant
algebroids and omni-Lie algebroids.

\subsection{Representation up to homotopy of Lie algebras and their semidirect products}

$L_\infty$-algebras,  sometimes  called strongly homotopy Lie
algebras,  were introduced by Drinfeld and Stasheff
\cite{stasheff:shla} as a model for ``Lie algebras that satisfy
Jacobi up to all higher homotopies''. The following convention of
$L_\infty$-algebras has the same grading as in \cite{henriques} and
\cite{rw}.

\begin{defi}
An $L_\infty$-algebra is a graded  vector space $L=L_0\oplus
L_1\oplus\cdots$ equipped with a system $\{l_k|~1\leq k<\infty\}$ of
linear maps $l_k:\wedge^kL\longrightarrow L$ with degree
$\deg(l_k)=k-2$, where the exterior powers are interpreted in the
graded sense and the following relation with Koszul sign ``Ksgn'' is
satisfied for all $n\geq0$:
\begin{equation}
\sum_{i+j=n+1}(-1)^{i(j-1)}\sum_{\sigma}\sgn(\sigma)\Ksgn(\sigma)l_j(l_i(x_{\sigma(1)},\cdots,x_{\sigma(i)}),x_{\sigma(i+1)},\cdots,x_{\sigma(n)})=0,
\end{equation}
where the summation is taken over all $(i,n-i)$-unshuffles with
$i\geq1$.
\end{defi}

Let $n=1$, we have
$$
l_1^2=0,\quad l_1:L_{i+1}\longrightarrow L_i,
$$
which means that $L$ is a complex and we usually write $\dM=l_1$.
Let $n=2$, we have
$$
\dM l_2(x,y)=l_2(\dM x,y)+(-1)^pl_2(x,\dM y),\quad \forall~x\in L_p,
y\in L_q,
$$
which means that $\dM $ is a derivation with respect to $l_2$. We
usually view $l_2$ as a bracket $[\cdot,\cdot]$. However, it is not
a Lie bracket, the obstruction of Jacobi identity is controlled by
$l_3$: \begin{eqnarray} \nonumber
&&l_2(l_2(x,y),z)+(-1)^{(p+q)r}l_2(l_2(y,z),x)+(-1)^{qr+1}l_2(l_2(x,z),y)\\
&&=-\dM l_3(x,y,z)-l_3(\dM x,y,z)+(-1)^{pq}l_3(\dM
y,x,z)-(-1)^{(p+q)r}l_3(\dM z, x,y),
\end{eqnarray}
where $x\in L_p,~ y\in L_q,~z\in L_q$ and $l_3$ also satisfies
higher coherence laws.

In particular, if the $k$-ary brackets are zero for all $k>2$, we
recover the usual notion of {\bf differential graded Lie algebras}
(DGLA). If $L$ is concentrated in degrees $<n$, $L$ is called an
 {\bf $n$-term $L_\infty$-algebra}.

In this paper, we mainly consider
 2-term $L_\infty$-algebras, which are equivalent  to Lie 2-algebras
 given in \cite{baez:2algebras} by John Baez and Alissa Crans.   In
this special case, $l_4$ is always zero. Thus restricting the coherence
law satisfied by $l_3$ on degree-0, we obtain
\begin{equation}
l_3(l_2(x,y),z,w)+c.p.-\big(l_2(l_3(x,y,z),w)+c.p.\big)=0,\quad\forall~x,y,z,w\in
L_0.
\end{equation} Lie 2-algebras are categorified
 version of Lie algebras. In a Lie 2-algebra, the Jacobi identity
 is replaced by an isomorphism which is called the {\bf
 Jacobiator}. The Jacobiator  satisfies a certain law of its own.  Given a 2-term
 $L_\infty$-algebra $L_1\stackrel{\dM}{\longrightarrow}L_0$, the
 underlying 2-vector space of the corresponding Lie 2-algebra is
made up by $L_0$  as the vector space of objects and $L_0\oplus
 L_1$ as the vector space of morphism. Please see \cite[Theorem
 4.3.6]{baez:2algebras} for more details.

Recall from \cite{baez:2algebras} that a Lie 2-algebra is skeletal
if isomorphic objects are equal. Viewing a Lie 2-algebra as a 2-term
$L_\infty$-algebra, explicitly we have
 \begin{defi}\label{defi:skeletal L 2}
A 2-term $L_\infty$-algebra $L_1\stackrel{\dM}{\longrightarrow}L_0$
is called {\bf skeletal} if $\dM=0$.
 \end{defi}

They also prove the following theorem
\begin{thm}{\rm\cite{baez:2algebras}}\label{thm:skeletal Lie 2 a}  There is a
 one-to-one correspondence between 2-term skeletal $L_\infty$-algebra
 $L_1\stackrel{\dM}{\longrightarrow}L_0$ and quadruples
 $(\frkk_1,\frkk_2,\phi,\theta)$ where $\frkk_1$ is a Lie algebra, $\frkk_2$ is a
 vector space, $\phi$ is  a representation of $\frkk_1$ on $\frkk_2$ and
 $\theta$ is a 3-cocycle on $\frkk_1$ with values in $\frkk_2$.\end{thm}

We briefly recall that given a 2-term skeletal $L_\infty$-algebra
 $L_1\stackrel{\dM}{\longrightarrow}L_0$,
 $\frkk_1$ is $L_0$, $\frkk_2$ is $L_1$, the representation $\phi$ comes from $l_2$
 and the 3-cocycle $\theta$ is obtained from $l_3$.

 Please see \cite{abad-crainic:rep-homotopy, abad} for
the general theory of  representation up to homotopy of Lie
algebroids. In this paper we only consider the 2-term representation
up to homotopy of Lie algebras.
\begin{defi}{\rm\cite{abad-crainic:rep-homotopy}}
A 2-term representation up to homotopy of a Lie algebra $\frkg$
consists of
\begin{itemize}
\item[\rm 1.] A 2-term complex of vector spaces
$V_1\stackrel{\dM}{\longrightarrow}V_0$.

\item[\rm 2.] Two linear maps $\mu_i: \g \to End(V_i)$, which are compatible with $\dM$, i.e. for
any $X\in\frkg,~\xi\in V_1$, we have
\begin{equation}\label{eqn:d mu}
\dM\circ\mu_1(X)(\xi)=\mu_0(X)\circ\dM(\xi).
\end{equation}

\item[\rm 3.] A linear map $\nu:\wedge^2\frkg\to\Hom(V_0,V_1)$ such that
\begin{eqnarray}
\label{eqn:mu0}\mu_0[X_1,X_2]-[\mu_0(X_1),\mu_0(X_2)]&=&\dM\circ
\nu(X_1,X_2),\\
\label{eqn:mu1}\mu_1[X_1,X_2]-[\mu_1(X_1),\mu_1(X_2)]&=&
\nu(X_1,X_2)\circ\dM,
\end{eqnarray}
as well as
\begin{equation}\label{eqn:d k}
[\mu_0(X_1) + \mu_1(X_1),\nu(X_2,X_3)]+c.p.= \nu([X_1,X_2],X_3)+c.p.,
\end{equation}
where $c.p.$ stands for  cyclic permutation.
\end{itemize}
We usually write $\mu=\mu_0+\mu_1$ and denote a 2-term
representation up to homotopy of a Lie algebra $\frkg$ by
$(V_1\stackrel{\dM}{\longrightarrow}V_0,\mu,\nu)$.
\end{defi}

In \cite[Example 3.25]{abad-crainic:rep-homotopy}, the authors
proved that associated to any representation up to homotopy
$V_\bullet$ of a Lie algebra $\frkg$, one can form a new $L_\infty
$-algebra $\frkg\ltimes V_\bullet$, which is their semidirect
product. Here we make this construction explicit in the 2-term case.

Let $(V_1\stackrel{\dM}{\longrightarrow}V_0,\mu,\nu)$ be a 2-term
representation up to homotopy of $\frkg$, then we can form a new
2-term complex
$$
(\frkg\ltimes
V_\bullet,\dM):V_1\stackrel{\dM}{\longrightarrow}(\frkg\oplus V_0).
$$
Define  $l_2:\wedge^2(\frkg\ltimes
V_\bullet)\longrightarrow\frkg\ltimes V_\bullet $ by setting
\begin{equation}\label{2bracket}
\left\{\begin{array}{l}l_2(X+\xi,Y+\eta)=[X,Y]+\mu_0(X)(\eta)-\mu_0(Y)(\xi),\\
l_2(X+\xi,f)=\mu_1(X)(f),\\
l_2(f,g)=0,\end{array}\right.
\end{equation}
for any $X+\xi,Y+\eta\in\frkg\oplus V_0$ and $f,~g\in V_1$.  One
should note that $l_2$ is not a Lie bracket, but we have
$$
l_2(l_2(X+\xi,Y+\eta),Z+\gamma)+c.p.=\dM(\nu(X,Y)(\gamma))+c.p..
$$
Define $l_3:\wedge^3(\frkg\ltimes
V_\bullet)\longrightarrow\frkg\ltimes V_\bullet $ by setting:
\begin{equation}\label{3bracket}
l_3(X+\xi,Y+\eta,Z+\gamma)=-\nu(X,Y)(\gamma)+c.p.,
\end{equation} then we have
\begin{pro}\label{pro:Lie 2}
With the above notation, if
$(V_1\stackrel{\dM}{\longrightarrow}V_0,\mu,\nu)$ is a 2-term
representation up to homotopy of a Lie algebra $\frkg$, then
$(V_1\stackrel{\dM}{\longrightarrow}(\frkg\oplus V_0),l_2,l_3)$ is a
2-term $L_\infty$-algebra.
\end{pro}

\subsection{Example I: Courant algebroids $TM\oplus T^*M$}

Courant algebroids were first introduced in \cite{lwx} to study the
double of Lie bialgebroids. It is a vector bundle $E\longrightarrow
M$ equipped with a nondegenerate symmetric bilinear form
$\pairing{\cdot,\cdot}$ on the bundle, an antisymmetric bracket
$\Courant{\cdot,\cdot}$ on the section space $\Gamma(E)$ and a
bundle map $\rho:E\longrightarrow TM$ such that a set of axioms are
satisfied. It can be viewed \cite{royt} as a Lie 2-algebroid with a
``degree 2 symplectic form''. The first example  is the standard
Courant algebroid $(\huaT=TM\oplus
T^*M,\pairing{\cdot,\cdot},\Courant{\cdot,\cdot},\rho)$ associated
to a manifold $M$, where $\rho:\huaT\longrightarrow TM$ is the
projection, the canonical pairing $\pairing{\cdot,\cdot}$ is given
by
\begin{equation}\label{eqn:pair}
\pairing{X+\xi,Y+\eta}=\frac{1}{2}\big(\xi(Y)+\eta(X)\big),\quad\forall
~X,Y\in\frkX(M),~\xi,\eta\in\Omega^1(M),
\end{equation}
the antisymmetric bracket $\Courant{\cdot,\cdot}$ is given by
\begin{equation}\label{eqn:bracket}
\Courant{X+\xi,Y+\eta}\triangleq[X,Y]+L_X\eta-L_Y\xi+\frac{1}{2}d(\xi(Y)-\eta(X)),\quad\forall~X+\xi,~Y+\eta\in\Gamma(\huaT).
\end{equation}
It is not a Lie bracket, but we have
\begin{equation}\label{eqn:jacobi}
\Courant{\Courant{e_1,e_2},e_3}+c.p.=d T(e_1,e_2,e_3),\quad
\forall~e_1,e_2,e_3\in\Gamma(\huaT),
\end{equation}
where $T(e_1,e_2,e_3) $ is given by
\begin{equation}\label{T}
T(e_1,e_2,e_3)=\frac{1}{3}(\pairing{\Courant{e_1,e_2},e_3}+c.p.).
\end{equation}

Now we realize the section space of $\huaT$ as the semidirect
product of the Lie algebra $\frkX(M)$ of vector fields with the
natural 2-term deRham complex
\begin{equation}\label{eq:natural-cx}
C^\infty(M)\stackrel{d}{\longrightarrow}\Omega^1(M).
\end{equation}
For this we need to define a representation up to homotopy of $\frkX(M)$
on this complex. For any $X\in\frkX(M)$, define linear actions $\mu_0$ and $\mu_1$ by
\begin{eqnarray}
\label{C 1}\mu_0(X)(\xi)&\triangleq &\Courant{X,\xi}=L_X\xi-\half d(\xi(X)),\quad\forall ~\xi\in\Omega^1(M)\\
\label{C 2}\mu_1(X)(f)&\triangleq &\pairing{X,d
f}=\frac{1}{2}X(f),\quad\forall~f\in C^\infty(M).
\end{eqnarray}
Define $\nu:\wedge^2\frkX(M)\to\Hom(\Omega^1(M),C^\infty(M))$ by
\begin{equation}\label{C 3}
\nu(X,Y)(\xi)=T(X,Y,\xi),\quad\forall
~X,Y\in\frkX(M),~\xi\in\Omega^1(M). \end{equation}
\begin{pro}\label{pro:rep up to homotopy}
With the above notations,
$(C^\infty(M)\stackrel{d}{\longrightarrow}\Omega^1(M),\mu=\mu_0+\mu_1,\nu)$
 is a representation up to homotopy of the Lie algebra $\frkX(M)$.
\end{pro}
\pf For any $f\in C^\infty(M)$, we have
$$
\mu_0(X)(d f)=L_X df-\frac{1}{2}d X( f)=\frac{1}{2}d X( f),
$$
which implies $\mu_0\circ d=d\circ\mu_1$, i.e. $\mu_i$'s are compatible
with the differential $d$. By straightforward computations, we have
\begin{eqnarray*}
\mu_0[X,Y](\xi)-[\mu_0(X),\mu_0(Y)](\xi)&=&\Courant{\Courant{X,Y},\xi}+c.p.\\
&=&d T(X,Y,\xi) =d \big(\nu(X,Y)(\xi)\big),\\
\mu_1[X,Y](f)-[\mu_1(X),\mu_1(Y)](f)&=&\half[X,Y](f)-\frac{1}{4}\big(X(Y(f))-Y(X(f))\big)\\
&=&\frac{1}{4}[X,Y](f),\\
\nu(X,Y)(d f)&=&T(X,Y,d
f)=\frac{1}{3}\big(\half[X,Y](f)+\frac{1}{4}\big(X(Y(f))-Y(X(f))\big)\\
&=&\frac{1}{4}[X,Y](f),
\end{eqnarray*}
which implies (\ref{eqn:mu0}) and (\ref{eqn:mu1}). At last we need
to prove (\ref{eqn:d k}), which is obviously equivalent to
\begin{equation}\label{c temp}
\mu_1(X)T(Y,Z,\xi)-T(Y,Z,\mu_0(X)(\xi))+c.p.(X,Y,Z)=T([X,Y],Z,\xi)+c.p.(X,Y,Z).
\end{equation}
Observe that since $\mu_0(X)(\xi)=\Courant{X,\xi}$, we have
$$T(Y,Z,\mu_0(X)(\xi))+c.p.(X,Y,Z)+T([X,Y],Z,\xi)+c.p.(X,Y,Z)=T([X,Y],Z,\xi)+c.p.(X,Y,Z,\xi).$$
Furthermore, since for any $f\in \CWM$, $\mu_1(X)(f)=\pairing{X,df}$
and the cotangent bundle $T^*M$ is isotropic under the pairing
$(\ref{eqn:pair})$, we have
$$
\mu_1(X)T(Y,Z,\xi)+c.p.(X,Y,Z)=\pairing{X,dT(Y,Z,\xi)}+c.p.(X,Y,Z,\xi).
$$

Thus, (\ref{c temp}) is equivalent to
$$
\pairing{X,d T(Y,Z,\xi)}+c.p.(X, Y, Z, \xi)=T([X,Y],Z,\xi)+c.p.(X, Y,
Z, \xi),
$$
which holds by Lemma 4.5 in \cite{rw}. \qed\vspace{3mm}

By Proposition \ref{pro:Lie 2}, we have

\begin{cor}\label{cor:courant}
$\big(C^\infty(M) \xrightarrow{\dM=0\oplus d }(\frkX(M)\oplus\Omega^1(M)),l_2,l_3\big)$
is a 2-term $L_\infty$-algebra, where $l_2$ and $l_3$ are given by
(\ref{2bracket}) and (\ref{3bracket}), in which $\mu_0$, $\mu_1$ and
$\nu$ are given by (\ref{C 1}),~(\ref{C 2}) and (\ref{C 3})
respectively.
\end{cor}
\begin{rmk}\label{rmk:L infty}

In \cite{rw}  the authors proved  that  the sections of  a Courant
algebroid $(\huaC,\pairing{\cdot,\cdot},\Courant{\cdot,\cdot},\rho)$
form an $L_\infty$-algebra. In the case when $\huaC=\huaT$ the
standard Courant algebroid, the 2-term $L_\infty$-algebra is given
by
$$
\CWM \xrightarrow{\dM=0\oplus d } \frkX(M) \oplus
\Omega^1(M)=\Gamma(\huaT),
$$
with brackets  given by
$$
l_2(e_1,e_2)=\Courant{e_1,e_2},\quad l_2(e_1,f)=\pairing{e_1,\dM
f},\quad l_3(e_1,e_2,e_3)=-T(e_1,e_2,e_3), \quad l_{i\ge 4} =0,
$$
for any $e_1,e_2,e_3\in\Gamma(\huaT),~f\in \CWM$. Here $T$ is
defined by (\ref{T}). It is easy to verify that this is the same as
our 2-term $L_\infty$-algebra in Corollary \ref{cor:courant}.
\end{rmk}

We can also modify our complex \eqref{eq:natural-cx} to
$$
\Omega^1(M)\stackrel{\Id}{\longrightarrow}\Omega^1(M).
$$
Following the same procedure, we obtain another representation up to
homotopy of the Lie algebra $\frkX(M)$ and therefore obtain another
2-term $L_\infty$-algebra which is also totally determined by the
Courant algebroid
$(\huaT,\pairing{\cdot,\cdot},\Courant{\cdot,\cdot},\rho)$. More
precisely, for any $X\in \frkX(M),~\xi\in\Omega^1(M)$, $\mu_0=\mu_1$
is given by
$$
\mu_0(X)(\xi)\triangleq\Courant{X,\xi},
$$
and
$\nu:\frkX(M):\longrightarrow\Omega^2(\frkX^2(M),\End(\Omega^1(M),\Omega^1(M)))$
is given by
$$
\nu(X,Y)(\xi)\triangleq d T(X,Y,\xi).
$$
\begin{pro}
With the above notations,
$(\Omega^1(M)\stackrel{\Id}{\longrightarrow}\Omega^1(M),\mu=\mu_0=\mu_1,\nu)$
is a representation up to homotopy of the Lie algebra $\frkX(M)$.
\end{pro}

\subsection{Example \textup{II}: Courant algebroids over a point and
  string Lie 2-algebras} \label{sec:string}

A Courant algebroid over a point is literally  a quadratic Lie
algebra, namely a Lie algebra $\frkk$ together with nondegenerate
inner product $\pairing{\cdot,\cdot}$ which is invariant under the
adjoint action.  People often think that a Courant algebroid over a
point  is a string Lie 2-algebra \footnote{private conversation with
  John Baez and Urs Schreiber}. Here we justify this thinking.
\begin{defi}\label{defi:string-alg}
The  string Lie 2-algebra associated to a quadratic Lie algebra
$(\frkk, \pairing{\cdot,\cdot})$, is a  2-term $L_\infty$-algebra
$\mathbb R\stackrel{0}{\longrightarrow}\frkk$, whose degree-$0$
part is $\frkk$, degree-$1$ part is $\mathbb R$,  and $l_2,~l_3$ are
given by
\begin{eqnarray*}
~l_2(e,c)&=&0,\\
~l_2((e_1,c_1),(e_2,c_2))&=&([e_1,e_2],0),\\
~l_3((e_1,c_1),(e_2,c_2),(e_3,c_3))&=&(0,\pairing{[e_1,e_2],e_3}),
\end{eqnarray*}
where $e,~e_1,~e_2,~e_3\in\frkk,~c,~c_1,~c_2,~c_3\in\mathbb R$.
\end{defi}

The representation $\phi $ of $\frkk$ on $\mathbb R$ and the
3-cocycle $\theta:\wedge^3\frkk\longrightarrow \mathbb R$ in the
corresponding quadruple in Theorem \ref{thm:skeletal Lie 2 a} is
given by
\begin{eqnarray}
\nonumber\rho(e)(c)&=&l_2(e,c)\\
\label{3cocycle}\theta(e_1,e_2,e_3)&=&\pairing{[e_1,e_2],e_3}.
\end{eqnarray}

\begin{rmk}\label{rk:semisimple}
In the definition of string Lie 2-algebras \cite{baez:string,
  henriques} the base Lie algebra $\frkk$ is usually required
to be semisimple and of compact type, such that the Jacobiator gives
rise to the generator of $H^3(\frkk, \mathbb Z)=\mathbb Z$. This is
because  Witten's original motivation is to obtain a $3$-connected
cover of $Spin(n)$, and $\mathfrak{so}(n)$ is simple and of compact
type. However, to write down the structure of the string Lie
2-algebra, we only need a quadratic Lie algebra. Then $H^3(\frkk,
\mathbb Z)$ is not necessarily $\mathbb Z$ for a general quadratic
Lie algebra $\frkk$. For example, for the abelian Lie algebra $\R$,
any inner product is adjoint-invariant, and $H^3(\R, \Z)=0$.  We
thus face the danger that sometimes, the string Lie 2-algebra is
trivial, that is the Jacobiator corresponds to the trivial element
in $H^3(\frkk, \mathbb Z)$. Then what we have is a strict Lie
2-algebra, which is a crossed module of Lie algebras. Then the
integration of a crossed module of Lie algebras is simply a crossed
module of Lie groups.  However, we will verify that the example we
consider is not such a case.
\end{rmk}

The standard Courant algebroid
motivates us to consider the case of  the direct sum
$\frkk=\frkg\oplus\frkg^*$ of  a Lie algebra $\frkg$
and its dual with  the
semidirect product  Lie algebra structure:
$$
[X+\xi,Y+\eta]=[X,Y]_\g+\ad_X^*\eta-\ad_Y^*\xi,
$$here $[\cdot,\cdot]_\g$ is the Lie bracket of $\g$. The nondegenerate invariant pairing
$\pairing{\cdot,\cdot}$ on $\frkg\oplus \frkg^*$  is given by
$$
\pairing{X+\xi,Y+\eta}=\half(\eta(X)+\xi(Y)),\quad\forall~X+\xi,Y+\eta\in\frkg\oplus\frkg^*.
$$ With these definitions,
$(\frkg\oplus\frkg^*,[\cdot,\cdot],\pairing{\cdot,\cdot})$ is a
quadratic Lie algebra. In fact, we have
\begin{eqnarray*}
\pairing{[X_1+\xi_1,X_2+\xi_2],X_3+\xi_3}&=&\pairing{[X_1,X_2]_\g+\ad_X^*\xi_2-\ad_Y^*\xi_1,X_3+\xi_3}\\
&=&\pairing{[X_1,X_2]_\g,\xi_3}+c.p..
\end{eqnarray*}
Similarly, we have
$$
\pairing{X_2+\xi_2,[X_1+\xi_1,X_3+\xi_3]}=\pairing{[X_1,X_3]_\g,\xi_2}+c.p.,
$$
which implies that
$\pairing{X_2+\xi_2,[X_1+\xi_1,X_3+\xi_3]}+\pairing{[X_1+\xi_1,X_2+\xi_2],X_3+\xi_3}=0$,
i.e. the nondegenerate inner product $\pairing{\cdot,\cdot}$ is
invariant under the adjoint action. This example is a special case
of \cite[Theorem 1.12]{lu-weinstein} with $\g^*$ equipped with $0$
Lie bracket. Thus $(\g, \g^*)$ forms a Lie bialgebra or equivalently
$(\g\oplus \g^*, \g, \g^*)$ is a Manin triple. However, honestly we
have not found other Lie bialgebras (Manin triples) giving rise to
Lie 2-algebras of the form of semidirect products.

We  denote the corresponding string Lie 2-algebra of $(\g \oplus \g^*,
[\cdot, \cdot], \pairing{\cdot,\cdot})$ by
\begin{equation}\label{Lie 2 g g dual}
\mathbb R\stackrel{0}{\longrightarrow}\frkg\oplus \frkg^*.
\end{equation}
The corresponding 3-cocycle (see \eqref{3cocycle}), which we denote
by $\widetilde{\nu}:\wedge^3(\frkg\oplus \frkg^*)\longrightarrow
\mathbb R$ is
\begin{eqnarray}\label{eqn:nu3}
\widetilde{\nu}(X_1+\xi_1,X_2+\xi_2,X_3+\xi_3)&=&\pairing{[X_1+\xi_1,X_2+\xi_2],X_3+\xi_3}\\
\nonumber&=&\pairing{[X_1,X_2]_\g,\xi_3}+c.p.
\end{eqnarray}

\begin{pro}  $(\mathbb
R\stackrel{0}{\longrightarrow} \frkg^*,\mu_1=0,
\mu_0=\ad^*,\nu=[\cdot,\cdot]_\g)$ is a 2-term representation up to
homotopy of the Lie algebra $\frkg$. Moreover the string Lie
2-algebra $ \mathbb R\stackrel{0}{\longrightarrow}\frkg\oplus
\frkg^* $ is the semidirect product of $\frkg $ and the complex
$\mathbb R\stackrel{0}{\longrightarrow} \frkg^*$.
\end{pro}
\begin{proof}
Since $\dM=0$, we only need to verify that $\mu_0$ and $\mu_1$ are
Lie algebra morphisms and \eqref{eqn:d k}. $\ad^*: \g \to
\End(\g^*)$ and $0$ are both Lie algebra morphisms. \eqref{eqn:d k}
follows from Jacobi identity of $[\cdot,\cdot]_\g$.

Then  it is not hard to see that the string Lie 2-algebra $\mathbb
R\stackrel{0}{\longrightarrow}\g\oplus \frkg^*$ with formulas in
Definition \ref{defi:string-alg} is exactly the semidirect product
of $\g$ with the complex  $(\mathbb R\stackrel{0}{\longrightarrow}
\frkg^*,\mu_1=0, \mu_0=\ad^*,\nu=[\cdot,\cdot]_\g)$ with the
formulas \eqref{2bracket}, \eqref{3bracket}.
\end{proof}

\begin{pro}\label{pro:nondegenerate}
If the Lie algebra $\frkg$ is semisimple, the Lie algebra 3-cocycle
$\widetilde{\nu}$ which is given by (\ref{eqn:nu3}) is not exact.
\end{pro}
\pf Let $\pairing{\cdot,\cdot}_k$ be the Killing form on $\frkg$.
The proof follows from the fact that the Cartan 3-form
$\pairing{[\cdot,\cdot],\cdot}_k$   on a semisimple Lie algebra is
not exact. 
Since $\frkg$ is semisimple, the Killing form
$\pairing{\cdot,\cdot}_k$ is nondegenerate. Identify $\frkg^*$ and
$\frkg$ by using the Killing form $\pairing{\cdot,\cdot}_k$ and let
$\huaK$ be the corresponding isomorphism,
$$
\pairing{\huaK(\xi),X}_k=\pairing{\xi,X}.
$$
Assume that $\widetilde{\nu}=d\phi$ for some
$\phi:\wedge^2(\frkg\oplus\frkg^*)\longrightarrow \mathbb R$, define
$\varphi:\wedge^2(\frkg\oplus\frkg)\longrightarrow \mathbb R$ by
$$
\phi(X+\xi,Y+\eta)=\varphi(X+\huaK(\xi),Y+\huaK(\eta)),
$$then we have
\begin{eqnarray*}
\widetilde{\nu}(X,Y,\xi)&=&d\phi(X,Y,\xi)\\
&=&-\phi([X,Y],\xi)+\phi(\ad^*_X\xi, Y)-\phi(\ad^*_Y\xi, X)\\
&=&-\varphi([X,Y],\huaK(\xi))+\varphi([X,\huaK(\xi)],
Y)-\varphi([Y,\huaK(\xi)], X)\\
&=&d\varphi(X,Y,\huaK(\xi)).
\end{eqnarray*}
On the other hand, we have
$$
\widetilde{\nu}(X,Y,\xi)=\pairing{[X,Y],\xi}=\pairing{[X,Y],\huaK(\xi)}_k,
$$
which implies that the Cartan 3-from
$$
\pairing{[X,Y],\huaK(\xi)}_k=d\varphi(X,Y,\huaK(\xi)),
$$ is exact.
This is a contradiction. \qed\vspace{3mm}

In the last section, we will give  the integration of the string Lie
2-algebra $ \mathbb R\stackrel{0}{\longrightarrow}\frkg\oplus
\frkg^* $ by using the semidirect product of a Lie group with its
2-term representation up to homotopy. It turns out that this string
Lie 2-algebra can be integrated to a special Lie 2-group with a
finite dimensional model.

\subsection{Example III: Omni-Lie algebroids $\dev E\oplus \jet E$}

The notion of  omni-Lie algebroids, which was introduced by Chen and
Liu in \cite{clomni}, is a generalization of Weinstein's omni-Lie
algebras. Just as Dirac structures of an omni-Lie algebra
characterize Lie algebra structures on a vector space, Dirac
structures of an omni-Lie algebroid  characterize Lie algebroid
structures on a vector bundle. See \cite{clsdirac} for more details.
In fact, the role of omni-Lie algebroids $\dev E\oplus \jet E$ in
$E$-Courant algebroids, which was introduced in \cite{clsecourant},
is the same as the role of standard Courant algebroids $TM\oplus
T^*M$ in Courant algebroids.

Now we briefly recall the notion of omni-Lie algebroids and we will
see that it gives rise to a 2-term $L_\infty$-algebra which is a
semidirect product. In this subsection $E$ is a vector bundle over a
smooth manifold $M$, $\Gamma(E)$ is the section space of $E$.

For a vector bundle $E$, let $\dev E$ be its covariant differential
operator bundle. The associated Atiyah sequence is given by
\begin{equation}\label{Seq:DE}
\xymatrix@C=0.5cm{0 \ar[r] & \gl(E)  \ar[rr]^{\id} &&
                \dev{E}  \ar[rr]^{a} && TM \ar[r]  & 0.
                }
\end{equation}
The associated $1$-jet vector bundle ${\jet} E$ is defined as
follows. For any $m\in M$, $({\jet}{E})_m$ is defined as a quotient
of local sections of $E$. Two local sections $u_1$ and $u_2$ are
equivalent and we denote this by $u_1\sim u_2$ if
$$ u_1(m)=u_2(m) ~~\mbox{ and }~~ d\pairing{u_{1
},\xi}_m=d\pairing{u_{2 },\xi}_m, \quad\forall~ \xi\in\Gamma(E^*).$$
So any $\mu\in ({\jet}{E})_m$ has a representative $u\in\Gamma(E)$
such that
  $\mu=[u]_m$.
Let $\p$ be the projection which sends $[u]_m$ to $u(m)$.  Then
$\Ker\p \cong  \Hom(TM,E)$ and there is a short exact sequence,
called the jet sequence of $E$,
\begin{equation}\label{Seq:JetE}
\xymatrix@C=0.5cm{0 \ar[r] & \Hom(TM,E)  \ar[rr]^{~\qquad \e~} &&
                {\jet}{E} \ar[rr]^{\p} && E \ar[r]  & 0,
                }
\end{equation}
from which it is straightforward to see that $\jet E$ is a finite
dimensional vector bundle. Moreover, $\Gamma(\jet{E})$ is isomorphic
to $\Gamma(E) \oplus \Gamma (T^*M \otimes E)$ as an $\Real$-vector
space,
   and any
$u\in \Gamma (E)$ has a lift $\jetd u\in\Gamma ({\jet} E)$ by taking
its equivalence class, such that
\begin{equation}\label{Eqt:CWMmoduleGammaE}
\jetd(fu)=f \jetd u+ d f\otimes u, \quad \forall~ f\in \CWM.
\end{equation}
In \cite{clomni} the authors proved that
\begin{eqnarray*}  {\jet E} \cong
\set{\nu\in \Hom(\dev{E},E)\,|\,
\nu(\Phi)=\Phi\circ\nu(\Id_{E}),\quad\forall~ ~ ~\Phi\in \gl(E)}.
\end{eqnarray*}

Therefore, there is   an $E$-pairing between $\jet{E}$ and $\dev{E}$
by setting:
\begin{equation}
\label{Eqt:conpairingE} \conpairing{\mu,\frkd}_E~ \defbe
\frkd(u),\quad\forall~ ~~ \mu\in (\jet{E})_m,~\frkd\in(\dev{E})_m,
\end{equation}
where $u\in \Gamma(E)$ satisfies $\mu=[u]_m$. Particularly, one has
\begin{eqnarray}\label{conpairing1}
\conpairing{\mu,\Phi}_E &=& \Phi\circ \p(\mu),\quad\forall~ ~~ \Phi\in \gl(E),~\mu\in\jet{E};\\
\label{conpairing2} \conpairing{{\frky},\frkd}_{E} &=& {\frky}\circ
a(\frkd),\quad\forall~ ~~ \frky\in \Hom(TM,E),~\frkd\in\dev{E}.
\end{eqnarray}
 Furthermore, we claim that $\Gamma (\jet E)$ is an invariant
subspace of the Lie derivative $\Lied_{\frkd}$ for any
 $\frkd \in\Gamma(\dev{E})$, which is defined by the
 Leibniz rule as follows:
\begin{eqnarray}\nonumber
\conpairing{\Lied_{\frkd}\mu,\frkd\pie}_{E}&\defbe&
\frkd\conpairing{\mu,\frkd\pie}_{E}-\conpairing{\mu,[\frkd,\frkd\pie]_{\dev}}_{E},
\quad\forall~ \mu \in \Gamma(\jet{E}), ~
~\frkd\pie\in\Gamma(\dev{E}).
\end{eqnarray}

Introduce a nondegenerate symmetric $E$-valued $2$-form
$\ppairingE{\cdot,\cdot}$ on $ \omni\triangleq\dev{E}\oplus \jet{E}$
by:
$$
\ppairingE{\frkd+\mu,\frkr+\nu}\defbe \half(\conpairing{\frkd,\nu}_E
+\conpairing{\frkr,\mu}_E),\quad\forall~ ~~
\frkd,\frkr\in\dev{E},~\mu,\nu\in\jet{E}.
$$
Furthermore, we define an antisymmetric bracket
$\Courant{\cdot,\cdot}$ on $\Gamma(\omni)$ by:
\begin{eqnarray*}
\Courant{\frkd+\mu,\frkr+\nu}&\defbe&
[\frkd,\frkr]_{\dev}+\Lied_{\frkd}\nu-\Lied_{\frkr}\mu +
\half\big(\jetd\conpairing{\mu,\frkr}_E-\jetd\conpairing{\nu,\frkd}_E\big).
\end{eqnarray*}

In \cite{clomni} the authors  call  the quadruple
$(\omni,\Courant{\cdot,\cdot},\ppairingE{\cdot,\cdot},\omnirho)$ the
{\bf omni-Lie algebroid} \footnote{It is slightly different from the
notion given in \cite{clomni}, where the bracket is not
skewsymmetric.} associated to the vector bundle $E$,  where $\rho$
is the projection of $\omni$ onto $\dev{E}$. Even though
$\Courant{\cdot,\cdot}$ is antisymmetric,  it is not a Lie bracket.
More precisely, for any $X,Y,Z\in \Gamma(\omni)$, we have
$$\Courant{\Courant{X,Y},Z}+c.p.=\jetd T(X,Y,Z),$$
where $T:\Gamma(\wedge^3\omni)\longrightarrow \Gamma(E)$ is defined
by \be\label{T}
T(X,Y,Z)=\frac{1}{3}\big(\ppairingE{\Courant{X,Y},Z}+c.p.\big). \ee

Now let us construct a 2-term $L_\infty$-algebra from the omni-Lie
algebroid $\omni$. Obviously, $\Gamma(\dev E)$ is a Lie algebra and
there is a natural 2-term complex
$$
\Gamma(E)\stackrel{\dM=0\oplus\jetd}{\longrightarrow}\Gamma(\jet E).
$$

For any $\frkd\in\Gamma(\dev E)$, define linear actions $\mu_0$ and
$\mu_1$ by
\begin{equation}\label{O 1}\left\{\begin{array}{l}
\mu_0(\frkd)(\mu)\triangleq \Courant{\frkd,\mu}=\Lied_\frkd\mu-\jetd\ppairingE{\mu,\frkd},\quad\forall~\mu\in \Gamma(\jet E),\\
\mu_1(\frkd)(u)\triangleq \ppairingE{\frkd,\jetd
u}=\frac{1}{2}\frkd(u),\quad\forall~u\in\Gamma(E).\end{array}\right.
\end{equation}
Define $\nu:\wedge^2\Gamma(\dev E)\longrightarrow\Hom(\Gamma(\jet
E),\Gamma( E))$ by
\begin{equation}\label{O 2}
\nu(\frkd,\frkt)(\xi)=T(\frkd,\frkt,\mu),\quad\forall
~\frkd,\frkt\in\Gamma(\dev E),~\mu\in\Gamma(\jet E).
\end{equation}
Similar to Proposition \ref{pro:rep up to homotopy}, we prove
\begin{pro}
With the above notations,
$(\Gamma(E)\stackrel{\dM=0\oplus\jetd}{\longrightarrow}\Gamma(\jet
E),\mu=\mu_0+\mu_1,\nu)$
 is a representation up to homotopy of the Lie algebra $\Gamma(\dev
E)$.
\end{pro}
\begin{cor}
$\big(\Gamma(E)\stackrel{\dM=0\oplus\jetd}{\longrightarrow}\Gamma(\dev
E)\oplus\Gamma(\jet E),l_2,l_3\big)$ is a 2-term $L_\infty$-algebra,
where $l_2$ and $l_3$ are given by (\ref{2bracket}) and
(\ref{3bracket}), in which $\mu$ and $\nu$ are given by (\ref{O 1})
and (\ref{O 2}) respectively.
\end{cor}

\begin{rmk}
If the base manifold $M$ is a point, i.e. $E$ is a vector space, for
which  we use a new notation $V$, then $\dev V=\gl(V),~\jet V=V$,
and we recover the notion of omni-Lie algebras. The complex
$\Gamma(V)\stackrel{0\oplus\jetd}{\longrightarrow}\Gamma(\jet V)$
reduces to $V\stackrel{\Id}{\longrightarrow}V$, which is a
representation up to homotopy of $\gl(V)$ with $(\mu_0=\mu_1, \nu)$
given by
\begin{equation}\label{omni mu nu}
\mu_0(A)(u)=\half Au,\quad
\nu(A,B)=\frac{1}{4}[A,B],\quad\forall~A,B\in\gl(V),~u\in V.
\end{equation}
Hence even though an omni-Lie algebra   $\gl(V)\oplus V$ is not a
Lie algebra,  we can extend it to a  2-term $L_\infty$-algebra, of
which $L_0=\gl(V)\oplus V,~L_1=V$, $l_2$ and $l_3$ are given by
(\ref{2bracket}) and (\ref{3bracket}), in which $\mu$ and $\nu$ are
given by (\ref{omni mu nu}). This 2-term $L_\infty$-algebra is a
semidirect product of $\gl(V)$ with $V\xrightarrow{\Id} V$.

We will study the global object of the 2-term $L_\infty$-algebra
associated to an omni-Lie algebra in the forthcoming paper \cite{sheng-zhu:II}.
\end{rmk}

\section{Representation up to homotopy of Lie groups and semidirect
  product} \label{sec:gp}

The representation up to homotopy of a Lie group was introduced in
\cite{abad}. In this section we define the semidirect product of a
Lie group with a 2-term representation up to homotopy and prove that
the semidirect product is a Lie 2-group. Thus we first recall some
background on Lie 2-groups.

A group is a monoid where every element has an inverse. A 2-group is
a monoidal category where every object has a weak inverse and every
morphism has an inverse. Denote the category of smooth manifolds and
smooth maps by $\rm Diff$, a semistrict\footnote{Please see the
  introduction for the reason why we call it a semistrict Lie 2-group.} Lie 2-group is  a 2-group
in $\rm DiffCat$, where  $\rm DiffCat$ is the 2-category consisting
of categories, functors, and natural transformations in $\rm Diff$.
For more details, see \cite{baez:2gp} and here  we only recall the
expanded definition:
\begin{defi}{\rm\cite{baez:2gp}}
A semistrict Lie 2-group consists of an object $C$ in $\rm
DiffCat$, i.e. \[ \xymatrix{ C_1 \ar@<1ex>[r]^{s} \ar[r]_{t} & C_0},
\]
where $C_1,~C_0$ are objects in $\rm Diff$, $s,~t$ are the source
and target maps, and there is a vertical multiplication
$\cdot_\ve:C\times C\longrightarrow C$, together with
\begin{itemize}
\item[$\bullet$] a  functor (horizontal
multiplication) $\cdot_\h:C\times C\longrightarrow C$,
\item[$\bullet$] an identity object $1$,
\item[$\bullet$]a contravariant  functor  $\inv:C\longrightarrow C$
\end{itemize}
and the following natural isomorphisms:
\begin{itemize}
\item[$\bullet$] the {\bf associator}
$$
a_{x,y,z}:(x\cdot_\h y)\cdot_\h z\longrightarrow x\cdot_\h
(y\cdot_\h z),
$$
\item[$\bullet$]the {\bf left} and {\bf right unit}
$$
l_x:1\cdot_\h x\longrightarrow x,\quad r_x:x\cdot_\h
1\longrightarrow x,
$$
\item[$\bullet$]the  {\bf  unit} and {\bf  counit}
$$
i_x:1\longrightarrow x\cdot_\h \inv(x),\quad e_x:\inv(x)\cdot_\h
x\longrightarrow 1,
$$
\end{itemize}
such that the following diagrams commute:
\begin{itemize}
\item[$\bullet$] the {\bf pentagon identity } for the associator
\[
 \xymatrix{ & (w\cdot_\h x)\cdot_\h(y\cdot_\h z)\ar[dr]^{a_{w,x, y\cdot_\h z}}&  \\
((w\cdot_\h x)\cdot_\h y)\cdot_\h z\ar[ur]^{a_{(w\cdot_\h
x),y,z}}\ar[dr]_{a_{w,x,y}\cdot_\h 1_z}&&w\cdot_\h (x\cdot_\h
(y\cdot_\h z))\\
&(w\cdot_\h (x\cdot_\h y))\cdot_\h z\stackrel{a_{w,x\cdot_\h
y,z}}{\longrightarrow} w\cdot_\h ((x\cdot_\h y)\cdot_\h
z)\ar[ur]^{1_w\cdot_\h a_{x,y,z}}&}
\]
\item[$\bullet$] the {\bf triangle identity} for the left and right
unit lows:
\[
 \xymatrix{
( x\cdot_\h 1)\cdot_\h y\ar[rr]^{a_{x,1,y}}\ar[dr]^{r_x\cdot_\h 1_y}&&x\cdot_\h(1\cdot_\h y)\ar[dl]^{1_x\cdot_\h l_y}\\
&x\cdot_\h y& }
\]
\item[$\bullet$]the {\bf first zig-zag identity}:
\[
 \xymatrix{
&(x\cdot_\h \inv(x))\cdot_\h x\stackrel{a_{x,\inv(x),x}}{\longrightarrow}x\cdot_\h(\inv(x)\cdot_\h x)\ar[dr]^{1_x\cdot_\h e_x}&\\
1\cdot_\h x\ar[dr]^{l_x}\ar[ur]^{i_x\cdot_\h 1_x}&&x\cdot_\h 1\\
&x\ar[ur]^{r_x^{-1}}&}
\]
\item[$\bullet$]the {\bf second zig-zag identity}:
\[
 \xymatrix{
&\inv(x)\cdot_\h (x\cdot_\h \inv(x))\stackrel{a_{\inv(x),x,\inv(x)}}{\longrightarrow}(\inv(x)\cdot_\h x)\cdot_\h\inv(x)\ar[dr]^{e_x\cdot_\h 1_{\inv(x)}}&\\
\inv(x)\cdot_\h 1\ar[dr]^{r_{\inv(x)}}\ar[ur]^{1_{\inv(x)}\cdot_\h i_x}&&1\cdot_\h \inv(x).\\
&\inv(x)\ar[ur]^{l_{\inv(x)}^{-1}}&}
\]
\end{itemize}
\end{defi}

In the special case where $a_{x,y,z},~l_x,~r_x,~i_x,~e_x$ are all
identity isomorphisms, we call such a Lie 2-group a {\em strict Lie
  2-group}\footnote{The notion of strict Lie 2-groups is the same as \cite{baez:2gp}.}.

In the same reference, they also defined the notion of special
2-groups, which we recall here,
\begin{defi}\label{defi:special 2 G}
A {\bf special Lie 2-group}  is a Lie 2-group of which the source
and target coincide and  the left unit law $l$, the right unit law
$r$, the unit $i$ and the counit $e$ are identity isomorphisms.
\end{defi}

For classification of special Lie 2-groups,  we need the group cohomology with smooth cocycles, that is we
consider the cochain complex with smooth morphisms $G^{\times
n}\to M$ with $G$ a Lie group, $M$ its module. And the differential is defined as usual for
group cohomology. We denote this cohomology by $H_{sm}^\bullet(G, M)$.

\begin{thm}{\rm \cite[Theorem 8.3.7]{baez:2gp}}\label{thm:special 2 G}  There is a
one-to-one correspondence between special Lie 2-groups and
quadruples $(K_1,K_2,\Phi,\Theta)$ consisting a Lie group $K_1$, an
abelian group $K_2$, an action $\Phi$ of $K_1$ as automorphisms of
$K_2$ and a normalized smooth 3-cocycle $\Theta:K_1^3\longrightarrow
K_2$. Two special Lie 2-groups are isomorphic if and only if they
corresponds to the same\footnote{up to isomorphisms of groups of
  course} $(K_1, K_2, \Phi)$ and the
corresponding 3-cocycles represent the same element in $H^3_{sm}(K_1,
K_2)$.
\end{thm}
\begin{rmk}
We briefly recall that given a quadruple $(K_1,K_2,\Phi,\Theta)$ the
corresponding semistrict Lie 2-group has the Lie group $K_1$ as the space
of objects and the semidirect product Lie group $K_1 \ltimes_{\Phi}
K_2$ as the space of morphisms. The associator is given by $\Theta$.
\end{rmk}
\begin{defi}
A unital 2-term representation up to homotopy of a Lie group $G$
consists of
\begin{itemize}
\item[\rm 1.] A 2-term complex of vector spaces
$V_1\stackrel{\dM}{\longrightarrow}V_0$.

\item[\rm 2.] A nonassociative action $F_1$ on $V_0$ and $V_1$
  satisfying $$ \dM F_1 = F_1 \dM, \quad F_1(1_G)=\Id.$$

\item[\rm 3.] A smooth map $F_2:G\times G\longrightarrow\End(V_0,V_1)$ such that
\begin{equation}\label{eqn:F fail}
F_1(g_1)\cdot F_1(g_2)-F_1(g_1\cdot g_2)=[\dM,F_2(g_1,g_2)],
\end{equation}
as well as
\begin{equation}\label{eqn:F closed}
F_1(g_1)\circ F_2(g_2,g_3)-F_2(g_1\cdot g_2,g_3)+F_2(g_1, g_2\cdot
g_3)-F_2(g_1, g_2)\circ F_1(g_3)=0.
\end{equation}
\end{itemize}
\end{defi}
We denote this 2-term representation up to homotopy of the Lie group
$G$ by $(V_1\stackrel{\dM}{\longrightarrow}V_0,F_1,F_2)$. One should
be careful that even if $F_1$ is a usual associative action,
\eqref{eqn:F closed} is not equivalent to $F_2$ being a 2-cocycle. This is strangely different from the Lie algebra case (see Section
\ref{sec:int}). Define $\widetilde{F_2}:(G\ltimes
V_0)^3\longrightarrow V_1$ by
\begin{equation}\label{F2tuta}
\widetilde{F_2}((g_1,\xi_1),(g_2,\xi_2),(g_3,\xi_3))=F_2(g_1,
g_2)(\xi_3).
\end{equation}
If  $F_1$ is a usual associative action, we form  $G\ltimes V_0$
the semidirect product. Then
$V_1$ is a  $G\ltimes V_0$-module with an associated action
$\widetilde{F_1}$ of $G\ltimes V_0$ on $V_1$
$$
\widetilde{F_1}(g,\xi)(m)=F_1(g)(m),\quad\forall~m\in V_1.
$$

\begin{pro}\label{pro:3cocycle} If  $F_1$ is the usual associative action of the Lie group $G$  on the complex
$V_1\stackrel{\dM}{\longrightarrow}V_0$, then $\widetilde{F_2}$
defined by (\ref{F2tuta}) is a group 3-cocycle representing an
element in $H^3_{sm}(G\ltimes V_0,V_1)$.
\end{pro}
\pf By direct computations, we have
\begin{eqnarray*}
&&d\widetilde{F_2}((g_1,\xi_1),(g_2,\xi_2),(g_3,\xi_3),(g_4,\xi_4))\\&=&\widetilde{F_1}(g_1,\xi_1)\widetilde{F_2}((g_2,\xi_2),(g_3,\xi_3),(g_4,\xi_4))\\
&&-\widetilde{F_2}((g_1,\xi_1)\cdot(g_2,\xi_2),(g_3,\xi_3),(g_4,\xi_4))+\widetilde{F_2}((g_1,\xi_1),(g_2,\xi_2)\cdot(g_3,\xi_3),(g_4,\xi_4))\\
&&-\widetilde{F_2}((g_1,\xi_1),(g_2,\xi_2),(g_3,\xi_3)\cdot(g_4,\xi_4))+\widetilde{F_2}((g_1,\xi_1),(g_2,\xi_2),(g_3,\xi_3))\\
&=&F_1(g_1)F_2(g_2,g_3)(\xi_4)-F_2(g_1\cdot g_2,g_3)(\xi_4)+F_2(g_1,
g_2\cdot g_3)(\xi_4)\\
&&-F_2(g_1,g_2)(\xi_3+F_1(g_3)(\xi_4))+F_2(g_1,g_2)(\xi_3)\\
&=&\big(F_1(g_1)\circ F_2(g_2,g_3)-F_2(g_1\cdot g_2,g_3)+F_2(g_1,
g_2\cdot g_3)-F_2(g_1, g_2)\circ F_1(g_3)\big)(\xi_4).
\end{eqnarray*}
By (\ref{eqn:F closed}), $\widetilde{F_2}$ is a Lie group 3-cocycle.
\qed\vspace{3mm}

Similar to the fact that associated to any representation of a Lie
group, we can form a new Lie group which is their semidirect
product, for the 2-term representation up to homotopy of a Lie
group, we can form a Lie 2-group.

\begin{thm}\label{thm:main 1}
Given a 2-term representation up to homotopy
$(V_1\stackrel{\dM}{\longrightarrow}V_0,F_1,F_2)$ of a Lie group
$G$, its semidirect product with $G$ is defined to be
\begin{equation}\begin{array}{c}
G\times V_0\times V_1\\
\vcenter{\rlap{s }}~\Big\downarrow\Big\downarrow\vcenter{\rlap{t }}\\
G\times V_0.
 \end{array}\end{equation}
Then it is a Lie 2-group with the following structure maps:

 The source and target are given by
\be\label{s t} s(g,\xi,m)=(g,\xi),\quad t(g,\xi,m)=(g,\xi+\dM m).
\ee The vertical multiplication $\cdot_\ve$ is given by
$$
(h,\eta,n)\cdot_\ve(g,\xi,m) =(g,\xi,m+n),\quad \mbox{where}~
h=g,\eta=\xi+\dM m.
$$
The horizontal multiplication $\cdot_\h$ of objects is given by
\be\label{m o} (g_1,\xi)\cdot_\h (g_2,\eta)=(g_1\cdot
g_2,\xi+F_1(g_1)(\eta)), \ee the horizontal multiplication
$\cdot_\h$ of morphisms is given by \be\label{m m}
(g_1,\xi,m)\cdot_\h (g_2,\eta,n)=(g_1\cdot
g_2,\xi+F_1(g_1)(\eta),m+F_1(g_1)(n)). \ee The inverse map $\inv$ is
given by \be \inv(g,\xi)=(g^{-1},-F_1(g^{-1})(\xi)). \ee The
identity object is $(1_G,0)$.\\
The  associator
$$
a_{(g_1,\xi),(g_2,\eta),(g_3,\gamma)}:\big((g_1,\xi)\cdot_\h(g_2,\eta)\big)\cdot_\h(g_3,\gamma)\longrightarrow
(g_1,\xi)\cdot_\h\big((g_2,\eta)\cdot_\h(g_3,\gamma)\big)
$$
is given by \be\label{associator}
a_{(g_1,\xi),(g_2,\eta),(g_3,\gamma)}=(g_1\cdot g_2\cdot
g_3,\xi+F_1(g_1)(\eta)+F_1(g_1\cdot
g_2)(\gamma),F_2(g_1,g_2)(\gamma)). \ee The  unit
$i_{(g,\xi)}:(1_G,0)\longrightarrow (g,\xi)\cdot_\h \inv(g,\xi)$ is
given by \be\label{unit} i_{(g,\xi)}=(1_G,0,-F_2(g,g^{-1})(\xi)).
\ee All the other natural isomorphisms are identity isomorphisms.
\end{thm}
\pf By (\ref{s t}), (\ref{m o}) and (\ref{m m}), it is
straightforward to see that
\begin{eqnarray*}
s\big((g_1,\xi,m)\cdot_\h (g_2,\eta,n)\big)&=&s(g_1,\xi,m)\cdot_\h s(g_2,\eta,n),\\
t\big((g_1,\xi,m)\cdot_\h (g_2,\eta,n)\big)&=&t(g_1,\xi,m)\cdot_\h
t(g_2,\eta,n).
\end{eqnarray*}
Thus the multiplication $\cdot_\h$ respects  the source and target
map. Furthermore, it is not hard  to check that the horizontal and
vertical multiplications commute, i.e.
$$
\big((g,\xi+\dM m,n)\cdot_\h(g^\prime,\eta+\dM
p,q)\big)\cdot_\ve\big((g,\xi,m)\cdot_\h(g^\prime,\eta,p)\big)=\big((g,\xi+\dM
m,n)\cdot_\ve(g,\xi,m)\big)\cdot_\h\big((g^\prime,\eta+\dM
p,q)\cdot_\ve(g^\prime,\eta,p)\big)
$$

\be\label{m commute with v} \xymatrix@C+2em{
  \bullet &
  \ar@/_2pc/[l]_{(g,\xi)}_{}="0"
  \ar[l]|{(g,\xi+\dM m)}^{}="1"_{}="1b"
  \ar@/^2pc/[l]^{(g,\xi+\dM (m+n))}^{}="2"
  \ar@{=>} "0";"1"^{m}
  \ar@{=>} "1b";"2"^{n}
  \bullet &
  \ar@/_2pc/[l]_{(g^\prime,\eta)}_{}="3"
  \ar[l]|{(g^\prime,\eta+\dM p)}^{}="4"_{}="4b"
  \ar@/^2pc/[l]^{(g^\prime,\eta+\dM (p+q))}^{}="5"
  \ar@{=>} "3";"4"^{p}
  \ar@{=>} "4b";"5"^{q}
  \bullet.
} \ee

It follows from (\ref{eqn:F fail}) that the associator
$a_{(g_1,\xi),(g_2,\eta),(g_3,\gamma)}$ defined by
(\ref{associator}) is indeed a morphism from
$\big((g_1,\xi)\cdot_\h(g_2,\eta)\big)\cdot_\h(g_3,\gamma)$ to
$(g_1,\xi)\cdot_\h\big((g_2,\eta)\cdot_\h(g_3,\gamma)\big)$. To see that
it is natural, we need to show that \be\label{left} a_{(g_1,\xi+\dM
m),(g_2,\eta+\dM n),(g_3,\gamma+\dM k)}\cdot_\h
\Big(\big((g_1,\xi,m)\cdot_\h (g_2,\eta,n)\big)\cdot_\h
(g_3,\gamma,k)\Big) \ee is equal to \be\label{right}
\Big((g_1,\xi,m)\cdot_\h \big((g_2,\eta,n)\cdot_\h
(g_3,\gamma,k)\big)\Big)\cdot_\h
a_{(g_1,\xi),(g_2,\eta),(g_3,\gamma)}, \ee i.e. the following
 diagram commutates:
\[
\xymatrix{
\big((g_1,\xi)\cdot_\h(g_2,\eta)\big)\cdot_\h(g_3,\gamma)\ar[d]\ar[r]^{a}&(g_1,\xi)\cdot_\h\big((g_2,\eta)\cdot_\h(g_3,\gamma)\big)\ar[d]\\
\big((g_1,\xi+\dM m)\cdot_\h(g_2,\eta+\dM
n)\big)\cdot_\h(g_3,\gamma+\dM k)\ar[r]^{a}&(g_1,\xi+\dM
m)\cdot_\h\big((g_2,\eta+\dM n)\cdot_\h(g_3,\gamma+\dM k)\big). }
\]
By straightforward computations, we  obtain that (\ref{left}) is
equal to
$$
\big(g_1\cdot g_2\cdot g_3,\xi+F_1(g_1)(\eta)+F_1(g_1\cdot
g_2)(\gamma),m+F_1(g_1)(n)+F_1(g_1\cdot
g_2)(k)+F_2(g_1,g_2)(\gamma+\dM k)\big),
$$
and (\ref{right}) is equal to
$$
\big(g_1\cdot g_2\cdot g_3,\xi+F_1(g_1)(\eta)+F_1(g_1\cdot
g_2)(\gamma),m+F_1(g_1)(n)+F_1(g_1)\cdot
F_1(g_2)(k)+F_2(g_1,g_2)(\gamma)\big).
$$
 Hence (\ref{left}) is equal to (\ref{right}) by (\ref{eqn:F fail}).
This implies that $a_{(g_1,\xi),(g_2,\eta),(g_3,\gamma)}$ defined by
(\ref{associator}) is a natural isomorphism.

By (\ref{eqn:F fail}) and the fact that $F_1(1_G)=\Id$, the unit
given by (\ref{unit}) is indeed  a morphism from $(1_G,0)$ to
$(g,\xi)\cdot_\h \inv(g,\xi)$. To see that it is natural, we need to
prove
$$
\big((g,\xi,m)\cdot_\h\inv(g,\xi,m)\big)\cdot_\h
i_{(g,\xi)}=i_{(g,\xi+\dM m)},
$$
i.e. the following commutative diagram:

$$
\xymatrix{
  &  (1_G,0)\ar[dr]^{i_{(g,\xi)}} \ar_{i_{(g,\xi+\dM m)}}[dl]\\
(g,\xi+\dM m)\cdot_\h \inv(g,\xi+\dM m)
&&\ar[ll]_{\quad\qquad\qquad~(g,\xi,m)\cdot_\h\inv(g,\xi,m)}(g,\xi)\cdot_\h\inv(g,\xi)
}
$$
 This follows from
 $$
F_2(g,g^{-1})(\dM m)=F_1(g)\cdot F_1(g^{-1})(m)-F_1(g\cdot
g^{-1})(m)=F_1(g)\cdot F_1(g^{-1})(m)-m,
 $$ which is a special case of (\ref{eqn:F fail}).

Since $F(1_G)=\Id$, we have $$(1_G,0)\cdot_\h (g,\xi)=(g,\xi),\quad
(g,\xi)\cdot_\h(1_G,0)= (g,\xi).$$ Hence the left unit and the right
unit can also be taken as the identity isomorphism.

 The counit
$e_{(g,\xi)}:\inv(g,\xi)\cdot_\h (g,\xi)\longrightarrow (1_G,0)$ can
be taken as the identity morphism since we have
$$
\inv(g,\xi)\cdot_\h (g,\xi)=(g^{-1},-F_1(g^{-1})(\xi))\cdot_\h
(g,\xi)=(1_G,0).
$$
At last, we need to show
\begin{itemize}
\item[$\bullet$] the  pentagon identity  for the associator,
\item[$\bullet$] the  triangle identity for the left and right
unit laws,
\item[$\bullet$]the  first zig-zag identity,
\item[$\bullet$]the  second zig-zag identity.
\end{itemize}
We only give the proof of the  pentagon identity, the others can be
proved similarly and we leave them to the readers. In fact, the
pentagon identity is equivalent to
\begin{eqnarray*}
&&a_{(g_1,\xi),(g_2,\eta),(g_3,\gamma)\cdot_\h(g_4,\theta)}\cdot_\h
a_{(g_1,\xi)\cdot_\h(g_2,\eta),(g_3,\gamma),(g_4,\theta)}=\\
&&\big((g_1,\xi)\cdot_\h
a_{(g_2,\eta),(g_3,\gamma),(g_4,\theta)}\big)\cdot_\h
a_{(g_1,\xi),(g_2,\eta)\cdot_\h(g_3,\gamma),(g_4,\theta)}\cdot_\h
\big(a_{(g_1,\xi),(g_2,\eta),(g_3,\gamma)}\cdot_\h(g_4,\theta)\big)
\end{eqnarray*}
By straightforward computations, the left hand side is equal to
$$
\big(g_1\cdot g_2\cdot g_3\cdot g_4,\xi+F_1(g_1)(\eta)+F_1(g_1\cdot
g_2)(\gamma)+F_1(g_1\cdot g_2\cdot g_3)(\theta),F_2(g_1\cdot g_2,
g_3)(\theta)+F_2(g_1, g_2)(\gamma+F_1(g_3)(\theta))\big),
$$
and the right hand side is equal to
\begin{eqnarray*}
&\big(g_1\cdot g_2\cdot g_3\cdot g_4,\xi+F_1(g_1)(\eta)+F_1(g_1\cdot
g_2)(\gamma)+F_1(g_1\cdot g_2\cdot g_3)(\theta),\\&F_2(g_1,
g_2)(\gamma)+F_2(g_1, g_2\cdot g_3)(\theta)+F_1(g_1)\circ F_2(g_2,
g_3)(\theta)\big).
\end{eqnarray*}
By (\ref{eqn:F closed}), they are equal. \qed

\section{Integrating string Lie 2-algebra $\mathbb R\longrightarrow \frkg\oplus \frkg^*$}\label{sec:int}

As an application of Theorem \ref{thm:main 1}, we consider the
integration of the string Lie 2-algebra $\mathbb R\longrightarrow
\frkg\oplus \frkg^*$ given by (\ref{Lie 2 g g dual}). Now we
restrict to the case that $\frkg$ is finite dimensional. Obviously,
given a quadruple $(K_1,K_2,\Phi,\Theta)$ which represents a special
Lie 2-group (see Theorem \ref{thm:special 2 G}), by differentiation,
we obtain a quadruple $(\frkk_1,\frkk_2,\phi,\theta)$, which
represents a 2-term skeletal $L_\infty$-algebra.
\begin{defi}
A special Lie 2-group which is represented by
$(K_1,K_2,\Phi,\Theta)$ is an integration of a 2-term skeletal
$L_\infty$-algebra which is represented by
$(\frkk_1,\frkk_2,\phi,\theta)$ if the differentiation of
$(K_1,K_2,\Phi,\Theta)$ is $(\frkk_1,\frkk_2,\phi,\theta)$.
\end{defi}

If the differential $\dM$ in a 2-term complex
$V_1\stackrel{\dM}{\longrightarrow}V_0$ is 0, a representation up to
homotopy of Lie algebra $\frkg$ on
$V_1\stackrel{0}{\longrightarrow}V_0$ consists of two strict
representations $\mu_1$ and $\mu_0$, and  a liner map
$\nu:\frkg\wedge\frkg\longrightarrow\Hom(V_0,V_1)$ satisfying
equation (\ref{eqn:d k}).  This equation implies that $\nu$ is a Lie
algebra 2-cocycle representing an element in
$H^2(\frkg,\Hom(V_0,V_1))$, with the representation
$[\mu(\cdot),\cdot]$ of  $\frkg$ on $\Hom(V_0,V_1)$  defined by
$$
[\mu(\cdot),\cdot](X)(A)\triangleq[\mu(X), A]=\mu_1(X)\circ A-A\circ
\mu_0(X),\quad \forall ~X\in \frkg,~A\in\Hom(V_0,V_1).
$$

\begin{lem}
Define $\widetilde{\nu}:\wedge^3(\g\oplus V_0)\longrightarrow V_1$
by
$$
\widetilde{\nu}(X_1+\xi_1,X_2+\xi_2,X_3+\xi_3)=\nu(X_1,X_2)(\xi_3)+c.p.
$$
then $\nu$ is a 2-cocycle if and only if $\widetilde{\nu}$ is a
3-cocycle where the representation $\widetilde{\mu}$ of $\g\oplus
V_0$ on $V_1$ is given by
$$
\widetilde{\mu}(X+\xi)(m)=\mu(X)(m).
$$
\end{lem}
\pf By direct computations, for any $X_i+\xi_i\in\g\oplus
V_0,~i=1,2,3,4,$ we have
$$d\widetilde{\nu}(X_1+\xi_1,X_2+\xi_2,X_3+\xi_3,X_4+\xi_4)=d\nu(X_1,X_2,X_3)(\xi_4)+c.p.,$$
which implies the conclusion.\qed

 The  Lie algebra homomorphism $\mu$ from $\frkg$ to
$\End(V_0)\oplus\End(V_1)$ integrates to a Lie group homomorphism
$F_1$ from the simply connected Lie group $G$ of $\frkg$ to
$GL(V_0)\oplus GL(V_1)$ with
$$
\mu(X)=\frac{d}{dt}\Big|_{t=0}F_1(\exp tX),\quad\forall~X\in\frkg.
$$
 Consequently, $\Hom(V_0,V_1)$ is a $G$-module with $G$ action
$$
g\cdot A=F_1(g)\circ A \circ F_1(g)^{-1},\quad\forall ~g\in G,~A\in
\Hom(V_0,V_1).
$$ The Lie algebra 2-cocycle
$\nu:\frkg\wedge\frkg\longrightarrow\Hom(V_0,V_1)$ can integrate to
a smooth Lie group 2-cocycle $\overline{F_2}:G\times G\longrightarrow
\Hom(V_0,V_1)$, satisfying
\begin{equation}\label{eqn:F2 closed}
F_1(g_1)\circ(\overline{F_2})(g_2,g_3)\circ
F_1(g_1)^{-1}-(\overline{F_2})(g_1\cdot
g_2,g_3)+(\overline{F_2})(g_1, g_2\cdot
g_3)-(\overline{F_2})(g_1,g_2)=0,
\end{equation}
and $\overline{F_2}(1_G,1_G)=0.$ Let us explain how.

The classical theory of cohomology of discrete groups tells us that
the equivalence classes of extensions of $G$ by a $G$ module $M$,
one to one correspond to the elements in $H^2(G, M)$. In our case,
the same theory tells us that  $H^2_{sm}(G, \Hom(V_0, V_0))$
classifies  the equivalence classes of splitting extensions of $G$
by the $G$-module $\Hom(V_0, V_1)$, which is a splitting short exact
sequence of Lie groups, with $\Hom(V_0, V_1)$ endowed with an
abelian group structure,
\begin{equation}\label{eq:extension} \Hom(V_0, V_1) \to \hat{G} \to G. \end{equation}
In a general extension $\hat{G}$ is a principal bundle over $G$,
thus it usually does not permit a smooth lift $G
\xrightarrow{\sigma} \hat{G}$. It permits such a lift if and only if
the sequence splits. However in our case, since the abelian group
$\Hom(V_0, V_1) $ is a vector space, we have  $H^1(X,\Hom(V_0, V_1)
)=0$ for any manifold $X$. The proof makes use of a partition of
unity and similar to the proof showing that $H^1(X,
\underline{\R})=0$ for the sheaf cohomology. Hence all $ \Hom(V_0,
V_1)$ principal bundles are trivial. Therefore \eqref{eq:extension}
always splits. On the other hand
 it is well-known that when $G$ is simply connected, there is a
 one-to-one correspondence between extensions of $G$
 \cite[Theorem 4.15]{brahic} and extensions of its Lie algebra
 $\g$, which in turn are classified by the Lie algebra cohomology
 $H^2(\g, \Hom(V_0, V_1))$. Hence in our case the differentiation
 map
\[ H^2_{sm} (G, \Hom(V_0, V_1)) \to H^2(\g, \Hom(V_0, V_1)) \]
is an isomorphism. Hence $\nu$  always integrates to a smooth Lie
group 2-cocycle unique up to exact 2-cocycles. Then
$\overline{F_2}(1_G,1_G)=0$ can be arranged too, because we can
always modify the section $\sigma: G \to \hat{G}$ to satisfy
$\sigma(1_G)=1_{\hat{G}}$ and the modification of sections results
in an exact term. Then combined with (\ref{eqn:F2 closed}), it is
not hard to see that
\begin{equation}\label{temp1}
\overline{F_2}(1_G,g)=\overline{F_2}(g,1_G)=0,\quad\forall~g\in
G.\end{equation} Thus $\overline{F_2}$ is a normalized 2-cocycle.

\begin{pro}\label{pro:int mu nu }
For any 2-term representation up to homotopy $(\mu,\nu)$ of a Lie
algebra $\frkg$ on the complex
$V_1\stackrel{0}{\longrightarrow}V_0$, there is an associated
representation up to homotopy $(F_1,F_2)$ of the Lie group $G$ on
the complex $V_1\stackrel{0}{\longrightarrow}V_0$, where $F_1$ is
the integration of $\mu$ and $F_2:G\times G\longrightarrow
\End(V_0,V_1)$ is defined by
\begin{equation}\label{eqn:F2}
F_2(g_1,g_2)=\overline{F_2}(g_1,g_2)\circ F_1(g_1\cdot g_2).
\end{equation}
\end{pro}
\pf Obviously, (\ref{eqn:F fail}) is satisfied. To see (\ref{eqn:F
closed}) is also satisfied, combine  (\ref{eqn:F2}) with
(\ref{eqn:F2 closed}). By the fact that $F_1$ is a homomorphism, we
obtain
\begin{eqnarray*}
&F_1(g_1)\circ F_2(g_2,g_3)\circ F_1(g_2\cdot g_3)^{-1}\circ
F_1(g_1)^{-1}-F_2(g_1\cdot g_2,g_3)\circ F_1(g_1\cdot g_2\cdot
g_3)^{-1}\\&+F_2(g_1, g_2\cdot g_3)\circ F_1(g_1\cdot g_2\cdot
g_3)^{-1}-F_2(g_1,g_2)\circ F_1(g_1\cdot g_2)^{-1}=0.
\end{eqnarray*}
Composed with $F_1(g_1\cdot g_2\cdot g_3)$ on the right hand side,
we obtain (\ref{eqn:F closed}). \qed

By Proposition \ref{pro:int mu nu } and Theorem \ref{thm:main 1}, we
have
\begin{thm}\label{thm:main 2}
Let $G$ be the simply connected Lie group integrating $\g$, then the
string Lie 2-algebra $\mathbb
R\stackrel{0}{\longrightarrow}\frkg\oplus\frkg^*$ given by (\ref{Lie
2 g g dual}) integrates to the following Lie 2-group,
\begin{equation}\label{2 group}\begin{array}{c}
G\times \frkg^*\times \mathbb R\\
\vcenter{\rlap{s }}~\Big\downarrow\Big\downarrow\vcenter{\rlap{t }}\\
G\times \frkg^*,
 \end{array}\end{equation}
in which the source and target are given by \be
s(g,\xi,m)=t(g,\xi,m)=(g,\xi)\ee the vertical multiplication
$\cdot_\ve$ is given by
$$
(h,\eta,n)\cdot_\ve(g,\xi,m) =(g,\xi,m+n),\quad \mbox{where}~
h=g,\eta=\xi,
$$
the horizontal multiplication $\cdot_\h$ of objects is given by
$$
 (g_1,\xi)\cdot_\h (g_2,\eta)=(g_1\cdot
g_2,\xi+\Ad^*_{g_1}\eta),
$$
the horizontal multiplication $\cdot_\h$ of morphisms is given by
$$ (g_1,\xi,m)\cdot_\h (g_2,\eta,n)=(g_1\cdot
g_2,\xi+\Ad^*_{g_1}\eta,m+n),
$$
the inverse map $\inv$ is given by $$
\inv(g,\xi)=(g^{-1},-\Ad^*_{g^{-1}}\xi), $$ The
identity object is $(1_G,0)$,\\
the  associator
$$
a_{(g_1,\xi),(g_2,\eta),(g_3,\gamma)}:\big((g_1,\xi)\cdot_\h(g_2,\eta)\big)\cdot_\h(g_3,\gamma)\longrightarrow
(g_1,\xi)\cdot_\h\big((g_2,\eta)\cdot_\h(g_3,\gamma)\big)
$$
is given by \be\label{associator}
a_{(g_1,\xi),(g_2,\eta),(g_3,\gamma)}=(g_1\cdot g_2\cdot
g_3,\xi+\Ad^*_{g_1}\eta+\Ad^*_{g_1\cdot
g_2}\gamma,F_2(g_1,g_2)(\gamma)), \ee
All the other structures are identity isomorphisms.
\end{thm}
\pf Since $F_1$ is a usual associative action, we may modify the
unit (\ref{unit}) given in Theorem \ref{thm:main 1} to be the
identity natural transformation. It turns out that (\ref{2 group})
is a special Lie 2-group and is represented by $(G\ltimes
\frkg^*,\mathbb R,\Id,\widetilde{F_2})$, where $G\ltimes \frkg^*$ is
the semidirect product with the coadjoint action of $G$ on $\g^*$,
$\Id$ is the constant map $G\ltimes \g^* \to \Aut(\R)$ by mapping
everything to $\Id \in \Aut(\R)$, and  $\widetilde{F_2}$ is given by
\begin{equation}\label{eq:tf2}
\widetilde{F_2}((g_1,\xi_1),(g_2,\xi_2),(g_3,\xi_3))=F_2(g_1,
g_2)(\xi_3)=\overline{F_2}(g_1, g_2)\circ F_1(g_1\cdot g_2)(\xi_3).
\end{equation} Since $\overline{F_2}$ is normalized, $\widetilde{F_2}$ is also
normalized.  The string Lie 2-algebra $\mathbb
R\stackrel{0}{\longrightarrow}\frkg\oplus\frkg^*$ is  skeletal and
is represented by $(\frkg\oplus\frkg^*,\mathbb
R,0,\widetilde{\nu})$, where $\widetilde{\nu}$ is given by
(\ref{eqn:nu3}). Thus to show that our Lie 2-group is an integration
of the string Lie algebra, we only need to show that the
differential of the Lie group 3-cocycle $\widetilde{F_2}$ is the Lie
algebra 3-cocycle $\widetilde{\nu}$.  By direct computations
\cite[Lemma 7.3.9]{brylinski}, we have
\begin{eqnarray*}
&&\frac{\partial^3}{\partial_{t_1}\partial_{t_2}\partial_{t_3}}\Big|_{t_i=0}\sum_{\sigma\in
S_3}\epsilon(\sigma)\widetilde{F_2}\big((e^
{t_{\sigma(1)}X_{\sigma(1)}},t_{\sigma(1)}\xi_{\sigma(1)}),(e^
{t_{\sigma(2)}X_{\sigma(2)}},t_{\sigma(2)}\xi_{\sigma(2)}),(e^
{t_{\sigma(3)}X_{\sigma(3)}},t_{\sigma(3)}\xi_{\sigma(3)})\big)\\
&=&\frac{\partial^3}{\partial_{t_1}\partial_{t_2}\partial_{t_3}}\Big|_{t_i=0}\Big(\overline{F_2}(e^{
t_1X_1},e^ {t_2X_2})\circ F_1(e^{ t_1X_1}\cdot e^
{t_2X_2})(t_3\xi_3)\Big)+c.p.\\
&=&\frac{\partial^2}{\partial_{t_1}\partial_{t_2}}\Big|_{t_i=0}\Big(\overline{F_2}(e^{
t_1X_1},e^ {t_2X_2})\circ F_1(e^{ t_1X_1}\cdot e^ {t_2X_2})(\xi_3)\Big)+c.p.\\
&=&\frac{\partial}{\partial_{t_1}}\Big|_{t_1=0}\Big(\frac{\partial}{\partial_{t_2}}\big|_{t_2=0}\overline{F_2}(e^{
t_1X_1},e^ {t_2X_2})\circ F_1(e^{
t_1X_1})(\xi_3)+\overline{F_2}(e^{t_1X_1},1_G)\circ\frac{\partial}{\partial_{t_2}}\big|_{t_2=0}F_1(e^{
t_1X_1}\cdot e^{ t_2X_2})\Big)\\&&+c.p.\\
&=&\frac{\partial}{\partial_{t_1}}\frac{\partial}{\partial_{t_2}}\Big|_{t_i=0}\overline{F_2}(e^{
t_1X_1},e^
{t_2X_2})(\xi_3)+\frac{\partial}{\partial_{t_2}}\big|_{t_2=0}\overline{F_2}(1_G,e^
{t_2X_2})\circ \frac{\partial}{\partial_{t_1}}\Big|_{t_1=0}F_1(e^{
t_1X_1})(\xi_3)\\
&&+c.p.\quad\mbox{by (\ref{temp1})}\\
&=&\nu(X_1,X_2)(\xi_3)+c.p.\\
&=&\widetilde{\nu}(X_1+\xi_1,X_2+\xi_2,X_3+\xi_3)  \quad\mbox{by
(\ref{eqn:nu3})},
\end{eqnarray*}
which completes the proof. \qed
\begin{cor}
If Lie algebra $\frkg$ is semisimple, the Lie group 3-cocycle
$\widetilde{F_2}$ is not exact, i.e. $[\widetilde{F_2}]\neq0$
in $H^3_{sm}(G\ltimes \frkg^*,\mathbb R)$.
\end{cor}
\pf By Theorem \ref{thm:main 2}, the differentiation of the Lie
group 3-cocycle $\widetilde{F_2}$ is the Lie algebra 3-cocycle
$\widetilde{\nu}$. We only need to show that when $\frkg$ is
semisimple, the Lie algebra 3-cocycle $\widetilde{\nu}$ is not
exact. This fact is proved in Proposition \ref{pro:nondegenerate}.
\qed

\begin{remark}Since $G\ltimes \g^*$ is a fibration over $G$, the spectral sequence with $E_2^{p,q}=H^p_{sm}(G, H^q_{sm}( \g^*,
  \R))$ calculates the group cohomology $H^3_{sm}(G\ltimes
  \frkg^*,\mathbb R)$. Since $\g^*$ is an abelian group, $H^q_{sm}( \g^*,
  \R)=\wedge^q\g^*$. Thus when $G$ is compact, $ H^p_{sm}(G, H^q_{sm}( \g^*,
  \R))  = (\wedge^q\g^*)^G$ if $p=0$ and 0 otherwise, where
  $(\wedge^q\g^*)^G$ denotes the set of invariant elements of
  $\wedge^q \g^*$  under the coadjoint action of $G$. Thus when $G$ is compact, $H^3_{sm}(G\ltimes \frkg^*,\mathbb R)=(\wedge^3
g^*)^G \neq 0$ because the Cartan 3-form is an element of
$(\wedge^3 g^*)^G$.
\end{remark}

Notice that our 2-cocycle $\overline{F_2}$ is unique only up to
exact terms. Hence by Theorem \ref{thm:special 2 G},  to verify that our
construction is unique up to isomorphism, we need the following
lemma,

\begin{lemma}
If $\overline{F_2}= \dM \alpha$ is exact, then $\widetilde{F_2}=\dM \beta$ is also exact
with
\[\beta((g_1, \xi_1), (g_2, \xi_2))=\alpha(g_1)F_1(g_1)(\xi_2). \]
\end{lemma}
\begin{proof}
It is a direct calculation. Since $\overline{F_2}= \dM \alpha$, we have
\[\overline{F_2}(g_1, g_2) = F_1 (g_1) \alpha(g_2) F_1(g_1)^{-1} - \alpha(g_1
g_2) + \alpha(g_1). \]
From the definition of $F_2$, we know that
\[ F_2(g_1, g_2)= F_1(g_1) \alpha(g_2) F_1(g_1)^{-1} F_1(g_1 g_2) -
\alpha(g_1 g_2) F_1(g_1 g_2) + \alpha(g_1) F_1(g_1g_2). \]
By \eqref{eq:tf2}, we have
\[
\begin{split}
\widetilde{F_2}((g_1, \xi_1), (g_2, \xi_2),(g_3, \xi_3))= &F_1(g_1) \alpha(g_2)
F_1(g_1)^{-1} F_1(g_1g_2) (\xi_3) - \alpha(g_1g_2) F_1(g_1g_2) (\xi_3)
+ \alpha(g_1) F_1(g_1g_2) (\xi_3) \\
= &\dM \beta((g_1, \xi_1), (g_2, \xi_2),(g_3, \xi_3)),
\end{split}
 \] since $F_1$ is a group homomorphism.
\end{proof}

\begin{remark}
Our Lie 2-group as a stacky group has the underlying differential
stack $G \times \g^* \times B\R$. Thus it is  0,1,2-connected (i.e.
it has $\pi_0=\pi_1=\pi_2=0$) since $\pi_2(B\R)=\pi_1(\R)=0$ and
$\pi_1(B\R)=\pi_0(\R)=0$. Thus it is the unique 0,1,2-connected
stacky Lie group integrating the string Lie 2-algebra $\R
\xrightarrow{0} \g\oplus \g^*$ in the sense of \cite{z:lie2}.
\end{remark}
\bibliographystyle{habbrv}
\bibliography{../../bib/bibz}

\end{document}